\documentclass{birkjour}
\usepackage{amsfonts}
\usepackage{amsthm}
\usepackage{amssymb}
\usepackage{amsmath}
\usepackage{upref}
\usepackage{mathrsfs}

\newtheorem{thm}{Theorem}[section]

\newtheorem{prop}[thm]{Proposition}
\newtheorem{cor}[thm]{Corollary}
\theoremstyle{definition}

\theoremstyle{remark}
\newtheorem{rem}[thm]{Remark}

\numberwithin{equation}{section}

\newcommand{\dif}{\mathrm{d}}
\newcommand{\totdif}{\mathrm{D}}
\newcommand{\mf}{\mathscr{F}}
\newcommand{\mr}{\mathbb{R}}
\newcommand{\prst}{\mathbb{P}}
\newcommand{\stred}{\mathbb{E}}

\newcommand{\mn}{\mathbb{N}}

\newcommand{\mt}{\mathbb{T}}

\renewcommand{\labelenumi}{(\roman{enumi})}

\begin{document}
\title[Strong Solutions of Semilinear SPDE\MakeLowercase{s}]{Strong Solutions of Semilinear Stochastic\\ Partial Differential Equations}

\author{Martina Hofmanov\'a}

\address{Department of Mathematical Analysis\\ Faculty of Mathematics and Physics, Char\-les University\\ Sokolovsk\'a~83\\ 186~75 Praha~8\\ Czech Republic\vspace{2mm}}
\address{Institute of Information Theory and Automation of the ASCR\\ Pod~Vod\'arenskou v\v{e}\v{z}\'i~4\\ 182~08 Praha~8\\ Czech Republic\vspace{2mm}}
\address{IRMAR, ENS Cachan Bretagne, CNRS, UEB\\ av. Robert Schuman\\ 35~170 Bruz\\ France}

\email{martina.hofmanova@bretagne.ens-cachan.fr}

\thanks{This research was supported in part by the GA\,\v{C}R Grant no. P201/10/0752.}

\begin{abstract}
We study the Cauchy problem for a semilinear stochastic partial differential equation driven by a finite-dimensional Wiener process. In parti\-cular, under the hypothesis that all the coefficients are sufficiently smooth and have bounded derivatives, we consider the equation in the context of power scale generated by a strongly elliptic differential operator. Application of semigroup arguments then yields the existence of a continuous strong solution.
\end{abstract}

\maketitle

\section{Introduction}

In the present paper, we consider the following semilinear stochastic partial differential equation driven by a finite-dimensional Wiener process:
\begin{equation}\label{semilin}
\begin{split}
\dif u&=\big[\mathcal{A}u+F(u)\big]\dif t+\sigma(u)\,\dif W,\quad x\in\mt^N,\;t\in(0,T),\\
u(0)&=u_0,
\end{split}
\end{equation}
where $-\mathcal{A}$ is a strongly elliptic differential operator, $F$ is generally nonlinear unbounded operator and the diffusion coefficient in the stochastic term is also nonlinear.

It is a well known fact in the field of PDEs and SPDEs that many equations do not, in general, have classical or strong solutions. Unlike deter\-ministic problems, in the case of stochastic equations we can only ask whether the solution is smooth in the space variable. Thus, the aim of the present work is to determine conditions on coefficients and initial data under which there exists a spatially smooth solution to \eqref{semilin}. 

The literature devoted to the regularity for linear SPDEs is quite extensive mainly due to Krylov (see \cite{krylov1}), Krylov and Rozovskii (see \cite{krylov2}, \cite{krylov3} and the refe\-rences therein) and Flandoli (see \cite {flandoli}). However, there seems to be less papers concentrated on regularity for nonlinear SPDEs. A class of second order parabolic semilinear SPDEs was studied by Gy\"ongy and Rovira (see \cite{gyongy}) but they were only concerned with $L^p$-valued solutions. So our work can be regarded as an extension of their result.
Related problems were also discussed by Zhang (see \cite{zhang}, \cite{zhang1}), however, his assumptions are not satisfied in our case.

The main difficulty in the case of semilinear equations lies in the nonlinearities $F$ and $\sigma$ as, in higher order Sobolev spaces, we cannot expect the Lipschitz condition to be satisfied and hence the fixed point argument cannot be applied. In fact, even the linear growth condition does not hold true in general since the norm of a superposition does not grow linearly with the norm of the inner function (cf. Proposition \ref{admissible}, Corollary \ref{admissible1} and Remark \ref{first}).

In order to deal with \eqref{semilin}, we proceed in several steps. First of all, we consider the equation in $L^{p}$ and apply the Banach fixed point theorem to conclude the existence of an $L^p$-valued mild solution. Next, we study the Picard iterations as processes having values in Sobolev spaces ($W^{1,p}$ and afterwards $W^{m,p}$) and find suitable uniform estimates which remain valid also for the limit process.

As an immediate consequence of the main result, we obtain a continuous $C^{k,\lambda}$-valued solution. Here, we use the Sobolev embedding theorem so the stochastic integration in Banach spaces (see \cite{b2}, \cite{ondrejat3}), i.e. $W^{m,p}$, allows us to weaken the smoothness assumptions on coefficients.

The paper is organised as follows. In Section 2, we review the basic setting and state our main result. In Section 3, we collect important preliminary results related to superposition operators. In the final section, these results are applied and the proof of the main theorem is established.

This work was motivated by our research in the field of degenerate parabolic SPDEs of second order (see \cite{hof}), where smooth solutions of certain approximate nondegenerate problems were needed in order to derive the so-called kinetic formulation and to obtain kinetic solution. Nevertheless, since the regularity result of the present paper is based on properties of strongly elliptic operators, generalization to higher order equations does not cause any additional problems.

\section{Setting and main result}
\label{notation}

Let us first introduce the notation which will be used later on. We will consider perio\-dic boundary conditions: $x\in \mt^N$ where $\mt^N$ is the $N$-dimensional torus. The Sobolev spaces on $\mt^N$ will be denoted by $W^{m,p}(\mt^N)$ and by $W^{m,p}(\mt^N;\mr^n)$ we will denote the space of all functions $z=(z_1,\dots,z_n):\mt^N\rightarrow \mr^n$ such that $z_i\in W^{m,p}(\mt^N),\,i=1,\dots,n.$

We now give the precise assumptions on each of the terms appearing in the above equation \eqref{semilin}. We will work on a finite-time interval $[0,T],\,T>0.$
The operator $-\mathcal{A}$ is a strongly elliptic differential operator of order $2l$ with variable coefficients of class $C^\infty(\mt^N)$. Let us assume, in addition, that $-\mathcal{A}$ is formally symmetric and positive, i.e. we assume that $0$ belongs to the resolvent set of $-\mathcal{A}$. As an example of this operator let us mention for instance the second order differential operator in divergence form given by
$$\mathcal{A}u=\sum_{i,j=1}^N\partial_{x_i}\big(A_{ij}(x)\partial_{x_j}u\big),$$
where the coefficients $A_{ij}=A_{ji}$ are real-valued smooth functions and satify the uniform ellipticity condition, i.e. there exists $\alpha>0$ such that
$$\sum_{i,j=1}^N A_{ij}(x)\xi_i\xi_j\geq\alpha|\xi|^2,\qquad\; \forall x\in\mt^N,\;\;\forall\xi\in\mr^N.$$

Let us now collect basic facts concerning strongly elliptic differential operators satisfying our hypotheses (for a detailed exposition we refer the reader to \cite{pazy}).
Set $D(\mathcal{A}_p)=W^{2l,p}(\mt^N)$. Then the linear unbounded operator $\mathcal{A}_p$ in $L^p(\mt^N)$ defined by
$$\mathcal{A}_p u=\mathcal{A}u,\qquad u\in D(\mathcal{A}_p),$$
is the infinitesimal generator of a bounded analytic semigroup on $L^p(\mt^N)$. Let us denote this semigroup by $\mathcal{S}_p$. Fractional powers of $-\mathcal{A}_p$ are well defined and their domains correspond to classical Sobolev spaces (see \cite[Section 10]{ama1}), i.e.
$$\Big(D\big((-\mathcal{A}_p)^\delta\big),\big\|(-\mathcal{A}_p)^\delta\cdot\big\|_{L^p(\mt^N)}\Big)\cong \big(W^{2l\delta,p}(\mt^N),\|\cdot\|_{W^{2l\delta,p}(\mt^N)}\big),\qquad\delta\geq0.$$
We will also make use of the following property of analytic semigroups (see \cite[Chapter 2, Theorem 6.13]{pazy}):
\begin{equation}\label{pazy12}
\begin{split}
\forall t>0\quad&\forall \delta>0\quad\text{the operator $\;(-\mathcal{A}_p)^\delta\mathcal{S}_p(t)\;$ is bounded and }\\
&\;\quad\quad\;\|(-\mathcal{A}_p)^\delta\mathcal{S}_p(t)\|_{L^p(\mt^N)}\leq C_{\delta,p}\, t^{-\delta}.
\end{split}
\end{equation}

The nonlinearity term $F$ is defined as follows: for any $p\in[2,\infty)$
\begin{equation*}
\begin{split}
F:L^p(\mt^N)&\longrightarrow W^{-2l+1,p}(\mt^N)\\
z&\longmapsto \sum_{|\alpha|\leq 2l-1}a_\alpha\,\totdif^\alpha f_\alpha(z),
\end{split}
\end{equation*}
where $a_\alpha\in\mr$ and the functions $f_\alpha,\,|\alpha|\leq 2l-1,$ are smooth enough (exact assumptions will be given later). Let us denote by $f$ the vector of functions $(f_\alpha\,;|\alpha|\leq 2l-1,\,a_\alpha\neq 0)$ and denote its length by $\gamma$.

Throughout this article we fix $(\Omega,\mf,(\mf_t)_{t\geq0},\prst)$, a stochastic basis with a complete, right-continuous filtration. Let $\mathcal{P}$ denote the predictable $\sigma$-algebra on $\Omega\times[0,T]$ associated with $(\mf_t)_{t\geq0}$. For simplicity we will only consider finite-dimensional noise, however, the result can be extended to the infinite-dimensional case. Let $\mathfrak{U}$ be a finite-dimensional Hilbert space and let $\{e_i\}_{i=1}^d$ be its orthonormal basis. The process $W$ is a $d$-dimensional $(\mf_t)$-Wiener process in $\mathfrak{U}$, i.e. it has an expansion of the form $W(t)=\sum_{i=1}^d W_i(t)\, e_i$, where $W_i,\,i=1,\dots,d,$ are mutually independent real-valued standard Wiener processes relative to $(\mf_t)_{t\geq 0}$. The diffusion coefficient $\sigma$ is then defined as 
\begin{equation*}
\begin{split}
\sigma(z):\mathfrak{U}&\longrightarrow L^p(\mt^N)\\
h&\longmapsto \sum_{i=1}^d\sigma_i(\cdot,z(\cdot))\langle e_i,h\rangle,\qquad z\in L^p(\mt^N),
\end{split}
\end{equation*}
where the functions $\sigma_1,\dots,\sigma_d$ satisfy the following linear growth condition
\begin{equation}\label{lrust}
\sum_{i=1}^d\big|\,\sigma_i(x,\xi)\big|^2\leq C\big(1+|\xi|^2),\qquad x\in\mt^N,\,\xi\in\mr.
\end{equation}

Since we are going to solve \eqref{semilin} in $L^p(\mt^N)$, for $p\in[2,\infty)$, we need to ensure the existence of the stochastic integral as an $L^p(\mt^N)$-valued process. Recall, that $L^p$ spaces, $p\in[2,\infty)$, as well as the Sobolev spaces $W^{k,p}$, $p\in[2,\infty),\,k\geq 0$, belong to a class of the so-called 2-smooth Banach spaces, which are well suited for stochastic It\^o integration. (A detailed construction of stochastic integral for processes with values in 2-smooth Banach spaces can be found in \cite{b2} or \cite{ondrejat3}.) Let us denote by $\gamma(\mathfrak{U};X)$ the space of all $\gamma$-radonifying operators from $\mathfrak{U}$ to a 2-smooth Banach space $X$.
We will show that $\sigma(z)\in\gamma(\mathfrak{U};L^p(\mt^N))$ for any $z\in L^p(\mt^N)$ and
$$\|\sigma(z)\|^2_{\gamma(\mathfrak{U};L^p(\mt^N))}\leq C\big(1+\|z\|^2_{L^p(\mt^N)}\big).$$
Note, that the following fact holds true:
\begin{equation}\label{radon}
\begin{split}
\forall s>0&\quad\exists C_s\in(0,\infty)\quad\forall \xi_1,\dots,\xi_n\text{ independent }\mathcal{N}(0,1)\text{-random variables}\\
&\forall x_1,\dots,x_n\in\mr\qquad \bigg(\stred\Big|\sum_{i=1}^nx_i\xi_i\Big|^s\bigg)^{\frac{1}{s}}=C_s\bigg(\sum_{i=1}^n x_i^2\bigg)^{\frac{1}{2}}.
\end{split}
\end{equation}
The proof is, by the way, easy: $\big(\sum_{i=1}^n x_i^2\big)^{-\frac{1}{2}}\sum_{i=1}^nx_i\xi_i$ is an $\mathcal{N}(0,1)$-random varia\-ble.
Let $\{\xi_i\}_{i=1}^d$ be a sequence of independent $\mathcal{N}(0,1)$-random varia\-bles, by the definition of a $\gamma$-radonifying norm, using \eqref{radon} and \eqref{lrust}
\begin{equation}\label{rr}
\begin{split}
\|\sigma(z)&\|^2_{\gamma(\mathfrak{U};L^p(\mt^N))}=\stred\Big\|\sum_{i=1}^d\xi_i\,\sigma(z)e_i\Big\|_{L^p(\mt^N)}^2=\stred\Big\|\sum_{i=1}^d\xi_i\,\sigma_i(\cdot,z(\cdot))\Big\|^2_{L^p(\mt^N)}\\
&\leq\bigg(\stred\Big\|\sum_{i=1}^d\xi_i\,\sigma_i(\cdot,z(\cdot))\Big\|^p_{L^p(\mt^N)}\bigg)^{\frac{2}{p}}=\bigg(\int_{\mt^N}\stred\Big|\sum_{i=1}^d\xi_i\,\sigma_i(y,z(y))\Big|^p\dif y\bigg)^{\frac{2}{p}}\\
&=C_p^2\bigg(\int_{\mt^N}\Big(\sum_{i=1}^d\big|\sigma_i(y,z(y))\big|^2\Big)^{\frac{p}{2}}\dif y\bigg)^{\frac{2}{p}}\leq C\bigg(\int_{\mt^N}\big(1+|z(y)|^2\big)^\frac{p}{2}\dif y\bigg)^\frac{2}{p}\\
&\leq C\big(1+\|z\|_{L^p(\mt^N)}^2\big)
\end{split}
\end{equation}
and the claim follows.

The main result of this paper is as follows:

\begin{thm}\label{smooth1}
Let $p\in[2,\infty),\,q\in(2,\infty),\,m\in\mn$. We suppose that
$$u_0\in L^q(\Omega;W^{m,p}(\mt^N))\cap L^{mq}(\Omega;W^{1,mp}(\mt^N))$$
and
$$f_\alpha\in C^{m}(\mr)\cap C^{2l-1}(\mr),\;\; |\alpha|\leq 2l-1;\;\quad\sigma_i\in C^m(\mt^N\times\mr),\;\;i=1,\dots,d,$$
have bounded derivatives up to order $m$.
Then there exists a solution to \eqref{semilin} which belongs to
$$L^q(\Omega;C([0,T];W^{m,p}(\mt^N)))$$
and the following estimate holds true
\begin{equation*}
\stred\sup_{0\leq t\leq T}\|u(t)\|_{W^{m,p}(\mt^N)}^q\leq C\big(1+\stred\|u_0\|_{W^{m,p}(\mt^N)}^q+\stred\|u_0\|_{W^{1,mp}(\mt^N)}^{mq}\big).
\end{equation*}

\end{thm}

\begin{cor}\label{smooth2}
Let $k\in\mn_0,\,q\in(2,\infty)$ and $u_0\in L^q(\Omega;C^{k+1}(\mt^N))$. Assume that
$$f_\alpha\in C^{k+1}(\mr)\cap C^{2l-1}(\mr),\;|\alpha|\leq 2l-1;\quad\;\sigma_i\in C^{k+1}(\mt^N\times\mr),\;i=1,\dots,d,$$
have bounded derivatives up to order $k+1.$
Then there exists a solution to \eqref{semilin} which belongs to
$$L^q(\Omega;C([0,T];C^{k,\lambda}(\mt^N))),\qquad \lambda\in(0,1).$$

\end{cor}

\section{Preliminaries}

For the reader's convenience we shall first restate the following auxiliary result which is taken from \cite[Theorem 5.2.5]{runst}.

\begin{prop}\label{admissible}
Let $m\in\mn,\,m\geq 2,\,p\in[1,\infty)$. Suppose that the function $G\in C^m(\mr)$ has bounded derivatives up to order $m$. If $f\in W^{m,p}(\mt^N)\cap W^{1,mp}(\mt^N)$ then the following estimate holds true
\begin{equation*}
\big\|G(f)\big\|_{W^{m,p}(\mt^N)}\leq C\big(1+\|f\|_{W^{1,mp}(\mt^N)}^m+\|f\|_{W^{m,p}(\mt^N)}\big)
\end{equation*}
with a constant independent of $f$.

\begin{proof}
Since $G$ has a linear growth we have
$$\|G(f)\|_{L^p(\mt^N)}\leq C\big(1+\|f\|_{L^p(\mt^N)}\big).$$
Next, we will employ the chain rule formula for partial derivatives of compositions:
\begin{equation*}\label{chain}
\begin{split}
\totdif^\gamma G(f(x))=\sum_{l=1}^{|\gamma|}\sum_{\substack{\alpha_1+\cdots+\alpha_l=\gamma\\|\alpha_i|\neq 0}}\!C_{\gamma,l,\alpha_1,\dots,\alpha_l}\,G^{(l)}(f(x))\,\totdif^{\alpha_1}f(x)\cdots\totdif^{\alpha_l}f(x),
\end{split}
\end{equation*}
where $\gamma=(\gamma_1,\dots,\gamma_N),\,\alpha_i=(\alpha_i^1,\dots,\alpha_i^N),\,i=1,\dots,l,$ are multiindices and $C_{\gamma,l,\alpha_1,\dots,\alpha_l}$ are certain combinatorial constants. It is sufficient to consider $|\gamma|=m$. By the H\"{o}lder inequality we obtain
\begin{equation*}
\begin{split}
\big\|G^{(l)}(f)\,\totdif^{\alpha_1}f\cdots\totdif^{\alpha_l}f\big\|_{L^p(\mt^N)}\leq \big\|G^{(l)}\big\|_{L^\infty(\mr)}\prod_{i=1}^l\big\|\totdif^{\alpha_i}f\big\|_{L^{\frac{mp}{|\alpha_i|}}(\mt^N)}.
\end{split}
\end{equation*}
Due to interpolation inequalities, we have
\begin{equation*}
\|f\|_{W^{|\alpha_i|,\frac{mp}{|\alpha_i|}}(\mt^N)}\leq C\|f\|_{W^{1,mp}(\mt^N)}^{1-\theta_i}\|f\|_{W^{m,p}(\mt^N)}^{\theta_i}\qquad\text{ with }\qquad\theta_i=\frac{|\alpha_i|-1}{m-1}.
\end{equation*}
Therefore
\begin{equation*}
\begin{split}
\big\|\totdif^\gamma G(f)\big\|_{L^p(\mt^N)}&\leq C\max_{1\leq l\leq m}\sum_{\substack{\alpha_1+\cdots+\alpha_l=\gamma\\|\alpha_i|\neq 0}}\,\prod_{i=1}^l\|f\|_{W^{1,mp}(\mt^N)}^{1-\theta_i}\,\|f\|_{W^{m,p}(\mt^N)}^{\theta_i}\\
&\leq C\max_{1\leq l\leq m}\|f\|_{W^{1,mp}(\mt^N)}^{l-\frac{m-l}{m-1}}\|f\|_{W^{m,p}(\mt^N)}^{\frac{m-l}{m-1}}\\
&\leq C\big(\|f\|_{W^{1,mp}(\mt^N)}^m+\|f\|_{W^{m,p}(\mt^N)}\big),
\end{split}
\end{equation*}
where we used the fact that $a^x(b/a)^\frac{m-x}{m-1}$ is monotone in $x$ so the maximal value is attained at $x=1$ or $x=m$. The proof is complete.
\end{proof}
\end{prop}

This result can be easily extended to more general outer function.

\begin{cor}\label{admissible1}
Let $m\in\mn,\,m\geq 2,\,p\in[1,\infty)$. Suppose that the function $G\in C^m(\mt^N\times\mr)$ has the linear growth
\begin{equation}\label{xx1}
|G(x,\xi)|\leq C(1+|\xi|),\qquad x\in\mt^N,\,\xi\in\mr,
\end{equation}
and bounded derivatives up to order $m$. If $f\in W^{m,p}(\mt^N)\cap W^{1,mp}(\mt^N)$ then the following estimate holds true
\begin{equation*}
\big\|G(\cdot,f(\cdot))\big\|_{W^{m,p}(\mt^N)}\leq C\big(1+\|f\|_{W^{1,mp}(\mt^N)}^m+\|f\|_{W^{m,p}(\mt^N)}\big)
\end{equation*}
with a constant independent of $f$.
\end{cor}

\begin{rem}\label{first}
The situation is much easier for the first order derivatives: fix $p\in[1,\infty)$ and let $f\in W^{1,p}(\mt^N)$
\begin{enumerate}
\item if $G\in C^1(\mr)$ with bounded derivative then
\begin{equation*}\label{xs1}
\big\|G(f)\|_{W^{1,p}(\mt^N)}\leq C\big(1+\|f\|_{W^{1,p}(\mt^N)}\big),
\end{equation*}
\item if $G\in C^1(\mt^N\times\mr)$ has the linear growth \eqref{xx1} and bounded derivative then
\begin{equation*}\label{xs2}
\big\|G(\cdot,f(\cdot))\|_{W^{1,p}(\mt^N)}\leq C\big(1+\|f\|_{W^{1,p}(\mt^N)}\big),
\end{equation*}
\end{enumerate}
where the constant $C$ is independent of $f$.
\end{rem}

\section{Proof of the main result}

Let us review the main ideas of the proof. The proof is divided into three steps. In the first step, we apply the Banach fixed point theorem to conclude the existence of an $L^p(\mt^N)$-valued mild solution of \eqref{semilin}. In the second step, we study Picard iterations of \eqref{semilin} and find a uniform estimate of the $W^{1,p}(\mt^N)$-norm. It is then used in the third step to derive a uniform estimate of the $W^{m,p}(\mt^N)$-norm. This estimate remains valid also for the limit process and the statement follows.

These steps will be formulated in the form of propositions.

\begin{prop}[Fixed point argument]\label{prop1}
Let $p,\,q\in[2,\infty)$. Assume that $u_0\in L^q(\Omega;L^p(\mt^N))$ and
$$f_\alpha\in C^{2l-1}(\mr),\;\; |\alpha|\leq 2l-1;\qquad\qquad\sigma_i\in C^1(\mt^N\times\mr),\;\;i=1,\dots,d,$$
have bounded derivatives of first order.
Then there exists a unique mild solution to \eqref{semilin} which belongs to
$$L^q(\Omega\times[0,T],\mathcal{P},\dif\prst\otimes\dif t;L^p(\mt^N)).$$

\begin{proof}

%{\color{red} Moreover, consider a function $f=(f_1,\dots,f_N):\mt^N\rightarrow\mr^N$ satisfying $f_i\in W^{m,p}(\mt^N),\,i=1,\dots,N,%\,m\in(0,\infty)$. Then there exists a constant $C_{m,p}>0$ such that
%$$\big\|(-\mathcal{A}_p)^{\frac{m}{2}-\frac{1}{2}}\diver f\big\|_{L^p(\mt^N)}\leq\big\|(-\mathcal{A}_p)^{\frac{m}{2}-\frac{1}{2}}\nabla %f\big\|_{L^p(\mt^N)}\leq C_{m,p}\big\|(-\mathcal{A}_p)^{\frac{m}{2}} f\big\|_{L^p(\mt^N)},$$
%where the operators $\nabla,(-\mathcal{A}_p)^{m-\frac{1}{2}},(-\mathcal{A}_p)^m$ applied to the vector function $f$ are acting %componentwise.
%Here, we used the equivalence between the graph norm associated with (fractional) powers of $-\mathcal{A}_p$ and the norm in %corresponding Sobolev spaces.
%$$\stred\int_0^T\bigg\|\int_0^t\mathcal{S}_p(t-s) F(v(s))\,\dif s\bigg\|_{L^p(\mt^N)}^p\dif t\\
%\leq C\stred\int_0^T\bigg(\int_0^t\big\|(-\mathcal{A}_p)^\frac{1}{2}\mathcal{S}_p(t-s) f(v(s))\big\|_{L^p(\mt^N)}\dif s\bigg)^p\dif t$$}
Let us denote
$$\mathscr{H}=L^q(\Omega\times[0,T],\mathcal{P},\dif\prst\otimes\dif t;L^p(\mt^N))$$
and define the mapping
\begin{equation*}
\begin{split}
\big(\mathscr{K}v\big)(t)&=\mathcal{S}_p(t)u_0+\int_0^t \mathcal{S}_p(t-s)F(v(s))\,\dif s+\int_0^t \mathcal{S}_p(t-s)\sigma(v(s))\,\dif W(s)\\
&=\mathcal{S}_p(t)u_0+\big(\mathscr{K}_1v\big)(t)+\big(\mathscr{K}_2v\big)(t),\qquad\quad t\in[0,T],\;v\in\mathscr{H}.
\end{split}
\end{equation*}
Here, we employ stochastic integration in $L^p(\mt^N)$ as introduced in Section \ref{notation}. 
We shall prove that $\mathscr{K}$ maps $\mathscr{H}$ into $\mathscr{H}$ and that it is a contraction.

Since $u_0\in L^q(\Omega;L^p(\mt^N))$ it follows easily that $\mathcal{S}(t)u_0\in\mathscr{H}$. 
In order to estimate the second term, let $\delta=\frac{2l-1}{2l}$ and note that
\begin{equation*}
\begin{split}
\mathcal{S}_p(t-s)F(v(s))&=\mathcal{S}_p(t-s)(-\mathcal{A}_p)^\delta(-\mathcal{A}_p)^{-\delta}\sum_{\substack{|\alpha|\leq 2l-1\\a_\alpha\neq 0}}a_\alpha\totdif^\alpha f_\alpha(v(s)),
\end{split}
\end{equation*}
where the operator $(-\mathcal{A}_p)^\delta$ commutes with the semigroup and the operator
\begin{equation*}
\begin{split}
\mathcal{B}_p\;:\;L^p(\mt^N;\mr^\gamma)&\longrightarrow L^p(\mt^N)\\
\{z_\alpha\}_{\alpha=1}^\gamma&\longmapsto(-\mathcal{A}_p)^{-\delta}\sum_{\substack{|\alpha|\leq 2l-1\\a_\alpha\neq 0}}a_\alpha\totdif^\alpha z_\alpha
\end{split}
\end{equation*}
is bounded. Indeed, the operators $L^r(\mt^N)\rightarrow L^r(\mt^N),\;v\mapsto a_\alpha\totdif^\alpha(-\mathcal{A}_r)^{-\delta} v,$ $|\alpha|\leq 2l-1,$ are clearly bounded.
%\begin{equation*}
%\begin{split}
%L^p(\mt^N)&\longrightarrow L^p(\mt^N)\\
%z&\longmapsto a_\alpha\totdif^\alpha(-\mathcal{A}_p)^{-\delta} z,\qquad|\alpha|\leq 2l-1,\;p\in[1,\infty),
%\end{split}
%\end{equation*}
%are clearly bounded. 
If $p^*$ is the conjugate exponent to $p$ and $z\in L^p(\mt^N;\mr^\gamma)$ then
\begin{equation*}
\begin{split}
\bigg\|&(-\mathcal{A}_p)^{-\delta}\sum_{\substack{|\alpha|\leq 2l-1\\a_\alpha\neq 0}}a_\alpha\totdif^\alpha z_\alpha\bigg\|_{L^p(\mt^N)}\\
&\quad=\sup_{\substack{v\in L^{p^*}(\mt^N)\\ \|v\|_{L^{p^*}(\mt^N)}\leq 1}}\Bigg|\int_{\mt^N}(-\mathcal{A}_p)^{-\delta}\sum_{\substack{|\alpha|\leq 2l-1\\a_\alpha\neq 0}}a_\alpha\totdif^\alpha z_\alpha(x)\, v(x)\,\dif x\Bigg|\\
&\quad=\sup_{\substack{v\in L^{p^*}(\mt^N)\\ \|v\|_{L^{p^*}(\mt^N)}\leq 1}}\Bigg|\sum_{\substack{|\alpha|\leq 2l-1\\a_\alpha\neq 0}}\int_{\mt^N}z_\alpha(x)\,a_\alpha\totdif^\alpha(-\mathcal{A}_{p^*})^{-\delta} v(x)\,\dif x\Bigg|\\
&\quad=\sup_{\substack{v\in L^{p^*}(\mt^N)\\ \|v\|_{L^{p^*}(\mt^N)}\leq 1}}\Bigg|\int_{\mt^N}\bigg\langle z(x),\Big\{a_\alpha\totdif^\alpha(-\mathcal{A}_{p^*})^{-\delta} v(x)\Big\}_{\substack{|\alpha|\leq 2l-1\\a_\alpha\neq 0}}\bigg\rangle_{\mr^\gamma}\,\dif x\Bigg|\\
&\quad\leq \|z\|_{L^p(\mt^N;\mr^\gamma)}\sup_{\substack{v\in L^{p^*}(\mt^N)\\ \|v\|_{L^{p^*}(\mt^N)}\leq 1}}\Big\|\Big\{a_\alpha\totdif^\alpha(-\mathcal{A}_{p^*})^{-\delta} v\Big\}_{\substack{|\alpha|\leq 2l-1\\a_\alpha\neq 0}}\Big\|_{L^{p^*}(\mt^N;\mr^\gamma)}\\
&\quad\leq C\,\|z\|_{L^p(\mt^N;\mr^\gamma)}
\end{split}
\end{equation*}
and the claim follows.
Next, all $f_\alpha,\,|\alpha|\leq 2l-1,$ have bounded derivatives hence at most linear growth, so it holds for any $z\in L^p(\mt^N)$
\begin{equation}\label{mnmn}
\big\|f(z)\big\|_{L^p(\mt^N)}\leq C\big(1+\|z\|_{L^p(\mt^N)}\big).
\end{equation}
Indeed, using $p$-norm as an equivalent norm on Euclidean space $\mr^\gamma$
\begin{equation*}
\begin{split}
\big\|f(z)\big\|_{L^p(\mt^N)}^p&=\sum_{\substack{|\alpha|\leq 2l-1\\a_\alpha\neq 0}}\int_{\mt^N}|f_\alpha(z(x))|^p\dif x\leq\sum_{\substack{|\alpha|\leq 2l-1\\a_\alpha\neq 0}}C_\alpha\int_{\mt^N}\big(1+|z(x)|^p\big)\,\dif x\\
&\leq C\big(1+\|z\|_{L^p(\mt^N)}^p\big).
\end{split}
\end{equation*}
If $v\in\mathscr{H}$, then using the above remark, the fact \eqref{pazy12},
the estimate \eqref{mnmn} and the Young inequality for convolutions we obtain
\begin{equation}\label{as}
\begin{split}
\big\|&\mathscr{K}_1v\big\|_{\mathscr{H}}^q=\stred\int_0^T\bigg\|\int_0^t\mathcal{S}_p(t-s) F(v(s))\,\dif s\big\|_{L^p(\mt^N)}^q\dif t\\
&\quad\leq \stred\int_0^T\bigg(\int_0^t\bigg\|(-\mathcal{A}_p)^\delta\mathcal{S}_p(t-s)\mathcal{B}_pf(v(s))\big\|_{L^p(\mt^N)}\dif s\bigg)^q\dif t\\
&\quad\leq C\,\stred\int_0^T\bigg(\int_0^t\frac{1}{(t-s)^\delta}\big\|\mathcal{B}_pf(v(s))\big\|_{L^p(\mt^N)}\dif s\bigg)^q\dif t\\
&\quad\leq C\,\stred\int_0^T\bigg(\int_0^t\frac{1}{(t-s)^\delta}\big\|f(v(s))\big\|_{L^p(\mt^N)}\dif s\bigg)^q\dif t\\
&\quad\leq C\,\stred\int_0^T\bigg(\int_0^t\frac{1}{(t-s)^\delta}\big(1+\|v(s)\|_{L^p(\mt^N)}\big)\dif s\bigg)^q\dif t\\
&\quad\leq C\, T^{q(1-\delta)}\,\stred\int_0^T\big(1+\|v(s)\|_{L^p(\mt^N)}\big)^q\dif s=C\, T^{q(1-\delta)}\,\big(T+\|v\|^q_{\mathscr{H}}\big).
\end{split}
\end{equation}
Next, by the Burk\-holder-Davis-Gundy inequality for martingales with values in 2-smooth Banach spaces (see \cite{b1}, \cite{ondrejat3}), we have
\begin{equation}\label{xxx}
\begin{split}
\big\|\mathscr{K}_2v\big\|_{\mathscr{H}}^q&=\stred\int_0^T\bigg\|\int_0^t\mathcal{S}_p(t-s)\sigma(v(s))\dif W(s)\bigg\|_{L^p(\mt^N)}^q\dif t\\
&\leq C\int_0^T\stred\bigg(\int_0^t\big\|\mathcal{S}_p(t-s)\sigma(v(s))\big\|_{\gamma(\mathfrak{U};L^p(\mt^N))}^2\dif s\bigg)^\frac{q}{2}\,\dif t\\
&\leq C\,T^\frac{q-2}{2}\int_0^T\stred\int_0^t\big\|\sigma(v(s))\big\|^q_{\gamma(\mathfrak{U};L^p(\mt^N))}\dif s\,\dif t.
\end{split}
\end{equation}
The $\gamma$-radonifying norm can be computed using \eqref{radon} similarly as in \eqref{rr}. Let $\{\xi_i\}_{i=1}^d$ be a sequence of independent $\mathcal{N}(0,1)$-random variables
\begin{equation*}
\begin{split}
\big\|&\sigma(v(s))\big\|^q_{\gamma(\mathfrak{U};L^p(\mt^N))}=\bigg(\stred\Big\|\sum_{i=1}^d\xi_i\,\sigma_i(\cdot,v(s,\cdot))\Big\|_{L^p(\mt^N)}^2\bigg)^\frac{q}{2}\\
&\leq\bigg(\stred\Big\|\sum_{i=1}^d\xi_i\,\sigma_i(\cdot,v(s,\cdot))\Big\|_{L^p(\mt^N)}^p\bigg)^\frac{q}{p}\leq\bigg(\int_{\mt^N}\stred\Big|\sum_{i=1}^d\xi_i\,\sigma_i(y,v(s,y))\Big|^p\dif y\bigg)^\frac{q}{p}\\
&= C\bigg(\int_{\mt^N}\Big(\sum_{i=1}^d\big|\sigma_i(y,v(s,y))\big|^2\Big)^\frac{p}{2}\dif y\bigg)^\frac{q}{p}\leq C\bigg(\int_{\mt^N}\big(1+|v(s,y)|^p\big)\dif y\bigg)^\frac{q}{p}\\
&\leq C\big(1+\|v(s)\|_{L^p(\mt^N)}^q\big).
\end{split}
\end{equation*}
Therefore
\begin{equation*}
\begin{split}
\big\|\mathscr{K}_2v\big\|_{\mathscr{H}}^q&\leq C\,T^\frac{q-2}{2}\int_0^T\stred\int_0^t\big(1+\|v(s)\|_{L^p(\mt^N)}^q\big)\dif s\,\dif t\leq C\,T^\frac{q}{2}\big(T+\|v\|_{\mathscr{H}}^q\big).
\end{split}
\end{equation*}
We conclude that $\mathscr{K}(\mathscr{H})\subset\mathscr{H}$ for any $T>0$.

In order to show the contraction property of $\mathscr{K}$, we will follow the approach from \eqref{as} and use the Lipschitz continuity of $f$. Indeed, $f_\alpha,\,|\alpha|\leq l,$ have bounded derivatives so they are Lipschitz continuous so
\begin{equation*}
\begin{split}
\big\|f(z_1)-f(z_2)\|_{L^p(\mt^N)}\leq C\,\|z_1-z_2\|_{L^p(\mt^N)},\qquad z_1,\,z_2\in L^p(\mt^N),
\end{split}
\end{equation*}
can be proved as \eqref{mnmn}.
For  $v,w\in\mathscr{H}$
\begin{equation*}
\begin{split}
\big\|&\mathscr{K}_1v-\mathscr{K}_1w\big\|_{\mathscr{H}}^q=\stred\int_0^T\bigg\|\int_0^t\mathcal{S}_p(t-s) \Big(F(v(s))-F(w(s))\Big)\,\dif s\bigg\|_{L^p(\mt^N)}^q\dif t\\
&\quad\leq \stred\int_0^T\bigg(\int_0^t\big\|(-\mathcal{A}_p)^\delta\mathcal{S}_p(t-s)\mathcal{B}_p\big(f(v(s))-f(w(s))\big)\big\|_{L^p(\mt^N)}\dif s\bigg)^q\dif t\\
&\quad\leq C\,\stred\int_0^T\bigg(\int_0^t\frac{1}{(t-s)^\delta}\big\|\mathcal{B}_p\big(f(v(s))-f(w(s)\big)\big\|_{L^p(\mt^N)}\dif s\bigg)^q\dif t\\
&\quad\leq C\,\stred\int_0^T\bigg(\int_0^t\frac{1}{(t-s)^\delta}\big\|f(v(s))-f(w(s))\big\|_{L^p(\mt^N)}\dif s\bigg)^q\dif t\\
&\quad\leq C\,\stred\int_0^T\bigg(\int_0^t\frac{1}{(t-s)^\delta}\|v(s)-w(s)\|_{L^p(\mt^N)}\dif s\bigg)^q\dif t\\
&\quad\leq C\, T^{q(1-\delta)}\,\stred\int_0^T\|v(s)-w(s)\|_{L^p(\mt^N)}^q\dif s=C\, T^{q(1-\delta)}\|v-w\|_{\mathscr{H}}^q.
\end{split}
\end{equation*}
In the case of $\mathscr{K}_2$ we employ the same calculations as in \eqref{xxx} and the sequel.
\begin{equation*}
\begin{split}
\big\|&\mathscr{K}_2v-\mathscr{K}_2w\big\|_{\mathscr{H}}^q=\stred\int_0^T\bigg\|\int_0^t\mathcal{S}_p(t-s)\Big(\sigma(v(s))-\sigma(w(s))\Big)\dif W(s)\bigg\|_{L^p(\mt^N)}^q\!\!\dif t\\
&\qquad\leq C\int_0^T\stred\bigg(\int_0^t\big\|\mathcal{S}_p(t-s)\big(\sigma(v(s))-\sigma(w(s))\big)\big\|_{\gamma(\mathfrak{U};L^p(\mt^N))}^2\dif s\bigg)^\frac{q}{2}\dif t\\
&\qquad\leq C\,T^\frac{q-2}{2}\int_0^T\stred\int_0^t\big\|\sigma(v(s))-\sigma(w(s))\big\|^q_{\gamma(\mathfrak{U};L^p(\mt^N))}\dif s\,\dif t
\end{split}
\end{equation*}
For the $\gamma$-radonifying norm we have
\begin{equation*}
\begin{split}
\big\|&\sigma(v(s))-\sigma(w(s))\big\|_{\gamma(\mathfrak{U};L^p(\mt^N)}^q\\
&\qquad\leq \bigg(\stred\Big\|\sum_{i=1}^d\xi_i\big(\sigma_i(\cdot,v(s,\cdot))-\sigma_i(\cdot,w(s,\cdot))\big)\Big\|_{L^p(\mt^N)}^2\bigg)^\frac{q}{2}\\
&\qquad\leq \bigg(\stred\Big\|\sum_{i=1}^d\xi_i\big(\sigma_i(\cdot,v(s,\cdot))-\sigma_i(\cdot,w(s,\cdot))\big)\Big\|_{L^p(\mt^N)}^p\bigg)^\frac{q}{p}\\
&\qquad=C\bigg(\int_{\mt^N}\Big(\sum_{i=1}^d\big|\sigma_i(y,v(s,y))-\sigma_i(y,w(s,y))\big|^2\Big)^\frac{p}{2}\dif y\bigg)^\frac{q}{p}\\
&\qquad\leq C\,\|v(s)-w(s)\|_{L^p(\mt^N)}^q,
\end{split} 
\end{equation*}
where the last inequality follows from the fact that all $\sigma_i,\,i=1,\dots,d,$ have bounded derivatives therefore are Lipschitz continuous. We conclude
$$\big\|\mathscr{K}_2v-\mathscr{K}_2w\big\|_\mathscr{H}^q\leq C\, T^\frac{q}{2}\|v-w\|_\mathscr{H}^q.$$
Consequently
$$\big\|\mathscr{K}v-\mathscr{K}w\|_\mathscr{H}\leq C\,\big(T^{1-\delta}+T^\frac{1}{2}\big)\|v-w\|_\mathscr{H},$$
where the constant does not depend on $T$ and $u_0$. Therefore, if 
\begin{equation}\label{fixx}
C\,\big(T^{1-\delta}+T^\frac{1}{2}\big)<1
\end{equation}
then the mapping $\mathscr{K}$ has unique fixed point $u$ in $\mathscr{H}$ which is a mild solution of \eqref{semilin}. The condition on $T$ can be easily removed by considering the equation on intervals $[0,\tilde{T}],\,[\tilde{T},2\tilde{T}],\,\dots$  with $\tilde{T}$ satisfying \eqref{fixx}.
%Moreover, it has a continuous modification and is also a weak solution (see \cite{prato}).
\end{proof}
\end{prop}

The estimates from previous proposition can be improved in order to obtain a better regularity of $u$.

\begin{prop}[Estimate in $W^{1,p}(\mt^N)$]\label{prop2}
Let $p\in[2,\infty),\,q\in(2,\infty)$. Assume that $u_0\in L^q(\Omega;W^{1,p}(\mt^N))$ and
$$f_\alpha\in C^{2l-1}(\mr),\;\; |\alpha|\leq 2l-1;\qquad\qquad\sigma_i\in C^1(\mt^N\times\mr),\;\;i=1,\dots,d,$$
have bounded derivatives of first order. Then the mild solution of \eqref{semilin} belongs to
$$L^q(\Omega;C([0,T];W^{1,p}(\mt^N)))$$
and the following estimate holds true
\begin{equation}\label{sobolev1}
\stred\sup_{0\leq t\leq T}\|u(t)\|_{W^{1,p}(\mt^N)}^q\leq C\big(1+\stred\|u_0\|_{W^{1,p}(\mt^N)}^{q}\big).
\end{equation}

\begin{proof}
Recall that $u$ is the limit of Picard iterations: let $u^0(t)=u_0$ and for $n\in\mn$ define
\begin{equation*}\label{mild}
\begin{split}
u^n(t)&= \mathcal{S}_p(t)\,u_0+\int_0^t\mathcal{S}_p(t-s)\,F\big(u^{n-1}(s)\big)\,\dif s\\
&\;\quad+\int_0^t  \mathcal{S}_p(t-s)\,\sigma\big(u^{n-1}(s)\big)\,\dif W(s).
\end{split}
\end{equation*}
We will show
\begin{equation}\label{kaka}
\stred\sup_{0\leq t\leq T}\|u^{n}(t)\|_{W^{1,p}(\mt^N)}^q\leq C\big(1+\stred\|u_0\|_{W^{1,p}(\mt^N)}^q\big),\qquad\forall n\in\mn,
\end{equation}
with a constant $C$ independent of $n$.
By induction on $n$, assume that the hypothesis is satisfied for $u^{n-1}$ and compute the estimate for $u^n$.
We will proceed term by term and follow the ideas of Proposition \ref{prop1}. Consider the operators $\mathcal{S}_p(t),\,t\geq0,$ restricted to the Sobolev space $W^{1,p}(\mt^N)$ and denote them by $\mathcal{S}_{1,p}(t),\,t\geq0$. These operators form a bounded analytic semigroup on $W^{1,p}(\mt^N)$ generated by the part of $\mathcal{A}_p$ in $W^{1,p}(\mt^N)$ (see \cite[Theorem V.2.1.3]{ama}). Let us denote this generator by $\mathcal{A}_{1,p}$.
Therefore we have
\begin{equation*}
\begin{split}
\stred\sup_{0\leq t\leq T}\|\mathcal{S}_{p}(t)u_0\|_{W^{1,p}(\mt^N)}^q&=\stred\sup_{0\leq t\leq T}\|\mathcal{S}_{1,p}(t)u_0\|_{W^{1,p}(\mt^N)}^q\\
&\leq C\,\stred\|u_0\|_{W^{1,p}(\mt^N)}^q.
\end{split}
\end{equation*}
As above, let $\delta=\frac{2l-1}{2l}$ and consider the operator
\begin{equation*}
\begin{split}
\mathcal{B}_{1,p}\;:\;W^{1,p}(\mt^N;\mr^\gamma)&\longrightarrow W^{1,p}(\mt^N)\\
\{z_\alpha\}_{\alpha=1}^\gamma&\longmapsto (-\mathcal{A}_{p})^{-\delta}\sum_{\substack{|\alpha|\leq 2l-1\\a_\alpha\neq0}}a_\alpha\totdif^\alpha z_\alpha.
\end{split}
\end{equation*}
We will show that it is a bounded operator. 
Indeed, according to Proposition \ref{prop1}, for any $z\in W^{1,p}(\mt^N;\mr^\gamma)$,
$$\big\|\mathcal{B}_{1,p}z\big\|_{L^p(\mt^N)}\leq C\|z\|_{L^p(\mt^N;\mr^\gamma)}.$$
For any multiindex $\beta=(\beta_1,\dots,\beta_N)$ such that $|\beta|=1$, we can write
\begin{equation*}
\begin{split}
\big\|\totdif^\beta\mathcal{B}_{1,p}z\big\|_{L^p(\mt^N)}=\Big\|\totdif^\beta(-\mathcal{A}_{p})^{-\frac{1}{2l}}(-\mathcal{A}_{p})^{-\frac{2l-1}{2l}+\frac{1}{2l}}\sum_{\substack{|\alpha|\leq 2l-1\\a_\alpha\neq0}}a_\alpha\totdif^\alpha z_\alpha\Big\|_{L^p(\mt^N)},
\end{split}
\end{equation*}
where the operator $L^p(\mt^N)\rightarrow L^p(\mt^N),\,v\mapsto\totdif^\beta(-\mathcal{A}_{p})^{-\frac{1}{2l}}v,\,$ is bounded. For each $\alpha,\,|\alpha|\leq 2l-1,$ let us fix a multiindex $\alpha'$ such that it is of order 1 and $\alpha-\alpha'$ is also a multiindex, i.e. $|\alpha'|=1$ and $|\alpha-\alpha'|=|\alpha|-1$. Note, that the operators $L^{r}(\mt^N)\rightarrow L^{r}(\mt^N),\,v\mapsto a_\alpha \totdif^{\alpha-\alpha'}(-\mathcal{A}_{r})^\frac{-2l+2}{2l}v,\,|\alpha|\leq 2l-1,$ are bounded as well. If $p^*$ is the conjugate exponent to $p$ we conclude
\begin{equation*}
\begin{split}
\Big\|&(-\mathcal{A}_{p})^{\frac{-2l+2}{2l}}\sum_{\substack{|\alpha|\leq 2l-1\\a_\alpha\neq0}}a_\alpha\totdif^\alpha z_\alpha\Big\|_{L^p(\mt^N)}\\
&\qquad=\sup_{\substack{v\in L^{p^*}(\mt^N)\\ \|v\|_{L^{p^*}(\mt^N)}\leq 1}}\Bigg|\int_{\mt^N}(-\mathcal{A}_{p})^{\frac{-2l+2}{2l}}\sum_{\substack{|\alpha|\leq 2l-1\\a_\alpha\neq 0}}a_\alpha\totdif^\alpha z_\alpha(x)\, v(x)\,\dif x\Bigg|\\
&\qquad=\sup_{\substack{v\in L^{p^*}(\mt^N)\\ \|v\|_{L^{p^*}(\mt^N)}\leq 1}}\Bigg|\sum_{\substack{|\alpha|\leq 2l-1\\a_\alpha\neq 0}}\int_{\mt^N}\totdif^{\alpha'}z_\alpha(x)\,a_\alpha\totdif^{\alpha-\alpha'}(-\mathcal{A}_{p^*})^{\frac{-2l+2}{2l}} v(x)\,\dif x\Bigg|\\
&\qquad\leq \Big\|\big\{\totdif^{\alpha'}z_\alpha\big\}_{\substack{|\alpha|\leq 2l-1\\a_\alpha\neq 0}}\Big\|_{L^p(\mt^N;\mr^\gamma)}\\
&\qquad\quad\quad\quad\times\sup_{\substack{v\in L^{p^*}(\mt^N)\\ \|v\|_{L^{p^*}(\mt^N)}\leq 1}}\bigg\|\Big\{a_\alpha\totdif^{\alpha-\alpha'}(-\mathcal{A}_{p^*})^{\frac{-2l+2}{2l}} v\Big\}_{\substack{|\alpha|\leq 2l-1\\a_\alpha\neq 0}}\bigg\|_{L^{p^*}(\mt^N;\mr^\gamma)}\\
&\qquad\leq C\,\|z\|_{W^{1,p}(\mt^N;\mr^\gamma)}
\end{split}
\end{equation*}
and the claim follows.
Therefore, we have
\begin{equation*}
\begin{split}
\stred\sup_{0\leq t\leq T}&\bigg\|\int_0^t\mathcal{S}_{p}(t-s)F\big(u^{n-1}(s)\big)\,\dif s\bigg\|_{W^{1,p}(\mt^N)}^q\\
&\leq \stred\sup_{0\leq t\leq T}\bigg(\int_0^t\Big\|(-\mathcal{A}_{p})^\delta\mathcal{S}_{p}(t-s)\mathcal{B}_{1,p}f\big(u^{n-1}(s)\big)\Big\|_{W^{1,p}(\mt^N)}\dif s\bigg)^q\\
&\leq \stred\sup_{0\leq t\leq T}\bigg(\int_0^t\Big\|(-\mathcal{A}_{1,p})^\delta\mathcal{S}_{1,p}(t-s)\mathcal{B}_{1,p}f\big(u^{n-1}(s)\big)\Big\|_{W^{1,p}(\mt^N)}\dif s\bigg)^q\\
&\leq C\,\stred\sup_{0\leq t\leq T}\bigg(\int_0^t\frac{1}{(t-s)^\delta}\big\|f\big(u^{n-1}(s)\big)\big\|_{W^{1,p}(\mt^N)}\dif s\bigg)^q\\
&\leq C T^{q(1-\delta)}\stred\sup_{0\leq t\leq T}\big\|f\big(u^{n-1}(t)\big)\big\|_{W^{1,p}(\mt^N)}^q.
\end{split}
\end{equation*}
To deduce a similar estimate for the stochastic term, we need to consider stochastic integration in $W^{1,p}(\mt^N)$. It holds
\begin{equation*}
\begin{split}
\big\|\sigma\big(u^{n-1}(s)\big)&\big\|_{\gamma(\mathfrak{U};W^{1,p}(\mt^N))}^q=\bigg(\stred\Big\|\sum_{i=1}^d\xi_i\,\sigma_i\big(\cdot,u^{n-1}(s,\cdot)\big)\Big\|_{W^{1,p}(\mt^N)}^2\bigg)^\frac{q}{2}\\
&\leq \bigg(\stred\Big\|\sum_{i=1}^d\xi_i\,(-\mathcal{A}_{p})^\frac{1}{2l}\sigma_i\big(\cdot,u^{n-1}(s,\cdot)\big)\Big\|_{L^p(\mt^N)}^p\bigg)^\frac{q}{p}\\
&= C\bigg(\int_{\mt^N}\Big(\sum_{i=1}^d\big|(-\mathcal{A}_{p})^\frac{1}{2l}\sigma_i\big(y,u^{n-1}(s,y)\big)\big|^2\Big)^\frac{p}{2}\dif y\bigg)^\frac{q}{p}\\\
&\leq C\sum_{i=1}^{d}\big\|\sigma_i\big(\cdot,u^{n-1}(s,\cdot)\big)\big\|^q_{W^{1,p}(\mt^N)}.
\end{split}
\end{equation*}
Since $q\in(2,\infty)$, we make use of the maximal estimate for stochastic convolution \cite[Corollary 3.5]{b1} which can be proved by the factorization method. For the reader's convenience we recall the basic steps of the proof.
According to the stochastic Fubini theorem \cite[Proposition 3.3(v)]{b2},
\begin{equation*}
\begin{split}
\int_0^t\mathcal{S}_p(t-s)\sigma\big(u^{n-1}(s)\big)\dif W(s)=\frac{1}{\Gamma(\alpha)}\int_0^t(t-s)^{\alpha-1}\mathcal{S}_p(t-s)\,y(s)\dif s,
\end{split}
\end{equation*}
where
\begin{equation*}
\begin{split}
y(s)=\frac{1}{\Gamma(1-\alpha)}\int_0^s(s-r)^{-\alpha}\mathcal{S}_p(s-r)\sigma\big(u^{n-1}(r)\big)\dif W(r).
\end{split}
\end{equation*}
Hence application of the H\"{o}lder, Burkholder-Davis-Gundy and Young inequalities yields (here the constant $C$ is independent on $T$)
\begin{equation*}
\begin{split}
\stred\sup_{0\leq t\leq T}\bigg\|\int_0^t\mathcal{S}_{p}(t-s)&\sigma\big(u^{n-1}(s)\big)\dif W(s)\bigg\|_{W^{1,p}(\mt^N)}^q\\
&\leq CT^{\frac{q}{2}-1}\,\stred\int_0^T\big\|\sigma\big(u^{n-1}(s)\big)\big\|_{\gamma(\mathfrak{U};W^{1,p}(\mt^N))}^q\dif s\\
\end{split}
\end{equation*}
so
\begin{equation*}
\begin{split}
\stred\sup_{0\leq t\leq T}\bigg\|\int_0^t\mathcal{S}_{p}(t-s)&\sigma\big(u^{n-1}(s)\big)\dif W(s)\bigg\|_{W^{1,p}(\mt^N)}^q\\
&\leq CT^{\frac{q}{2}-1}\sum_{i=1}^d\stred\int_0^T\big\|\sigma_i\big(\cdot,u^{n-1}(s,\cdot)\big)\big\|^q_{W^{1,p}(\mt^N)}\dif s\\
&\leq CT^{\frac{q}{2}}\sum_{i=1}^d\stred\sup_{0\leq t\leq T}\big\|\sigma_i\big(\cdot,u^{n-1}(t,\cdot)\big)\big\|^q_{W^{1,p}(\mt^N)}
\end{split}
\end{equation*}
and finally
\begin{equation*}\label{unifff}
\begin{split}
\stred\sup_{0\leq t\leq T}\|u^n(t)&\|_{W^{1,p}(\mt^N)}^q\leq C\,\stred\|u_0\|_{W^{1,p}(\mt^N)}^q\\
&\quad+CT^{q(1-\delta)}\stred\sup_{0\leq t\leq T}\big\|f\big(u^{n-1}(t)\big)\big\|_{W^{1,p}(\mt^N)}^q\\
&\quad+CT^\frac{q}{2}\sum_{i=1}^d\stred\sup_{0\leq t\leq T}\big\|\sigma_i\big(\cdot,u^{n-1}(t,\cdot)\big)\big\|_{W^{1,p}(\mt^N)}^q,
\end{split}
\end{equation*}
where the constant does not depend on $n$. 
Now, we make use of Remark \ref{first} and obtain
\begin{equation*}
\begin{split}
\stred\sup_{0\leq t\leq T}\|u^n(t)&\|_{W^{1,p}(\mt^N)}^q\leq C\,\stred\|u_0\|_{W^{1,p}(\mt^N)}^q\\
&\quad\quad\quad+C\big(T^{q(1-\delta)}+T^\frac{q}{2}\big)\bigg(1+\stred\sup_{0\leq t\leq T}\|u^{n-1}(t)\|_{W^{1,p}(\mt^N)}^q\bigg).
\end{split}
\end{equation*}
Let us make an additional hypothesis: assume that $T$ is such that
\begin{equation}\label{lplp}
C_T=C\big(T^{q(1-\delta)}+T^\frac{q}{2}\big)<1.
\end{equation}
Denoting $K_n=\stred\sup_{0\leq t\leq T}\|u^n(t)\|_{W^{1,p}(\mt^N)}^q,\,n\in\mn_0,$ we have
$$K_n\leq C\,\stred\|u_0\|_{W^{1,p}(\mt^N)}^q+C_T\big(1+K_{n-1}\big)$$
and inductively in $n$
\begin{equation}\label{mk}
\stred\sup_{0\leq t\leq T}\|u^n(t)\|_{W^{1,p}(\mt^N)}^q\leq \tilde{C}_T\big(1+\stred\|u_0\|_{W^{1,p}(\mt^N)}^q\big),
\end{equation}
where $\tilde{C}_T$ is independent $n$. So \eqref{kaka} follows if $T$ is sufficiently small.

In order to remove this condition, we consider a suitable partition of the interval $[0,T]$. Let $\tilde{T}>0$ satisfy \eqref{lplp} and $0<\tilde{T}<2\tilde{T}<\cdots<K\tilde{T}=T$ for some $K\in\mn$. Fix $k\in\{1,\dots,K\}$. We will study the processes $u^n,\,n\in\mn,$ on the interval $[(k-1)\tilde{T},k\tilde{T}]$ and find an estimate similar to \eqref{mk}. Each $u^n,\,n\in\mn,$ is the unique mild solution to the corresponding linear equation
\begin{equation*}
\begin{split}
\dif u^n&=\big[\mathcal{A}u^n+F\big(u^{n-1}\big)\big]\,\dif t+\sigma\big(u^{n-1}\big)\,\dif W,\qquad x\in\mt^N,\,t\in(0,T),\\
u(0)&=u_0.
\end{split}
\end{equation*}
Let $v(t,s,;u_0),\,t\geq s\geq0,$ be the mild solution of this problem with the initial condition $u_0$ given at time $s$. It follows from the uniqueness that for arbitrary $t\geq r\geq s\geq 0$
$$v\big(t,r;v(r,s;u_0)\big)=v(t,s;u_0)\qquad\prst\text{-a.s.}$$
and therefore we can write
\begin{equation*}
\begin{split}
u^n(t)=&\,\mathcal{S}_p\big(t-(k-1)\tilde{T}\big)u^n\big((k-1)\tilde{T}\big)+\int^t_{(k-1)\tilde{T}}\mathcal{S}_p(t-s)F\big(u^{n-1}(s)\big)\,\dif s\\
&+\int^t_{(k-1)\tilde{T}}\mathcal{S}_p(t-s)\sigma\big(u^{n-1}(s)\big)\,\dif W(s),\qquad t\in\big[(k-1)\tilde{T},T\big].
\end{split}
\end{equation*}
Following the same approach as above we obtain
$$\stred\sup_{(k-1)\tilde{T}\leq t\leq k\tilde{T}}\|u^n(t)\|_{W^{1,p}(\mt^N)}^q\leq \tilde{C}_{\tilde{T}}\Big(1+\stred\big\|u^n\big((k-1)\tilde{T}\big)\big\|_{W^{1,p}(\mt^N)}^q\Big)$$
with a constant similar to $\tilde{C}_{T}$ in \eqref{mk}. Hence
\begin{equation*}
\begin{split}
\stred&\sup_{(k-1)\tilde{T}\leq t\leq k\tilde{T}}\|u^n(t)\|_{W^{1,p}(\mt^N)}^q\\
&\quad\leq \tilde{C}_{\tilde{T}}\Big(1+\stred\sup_{(k-2)\tilde{T}\leq t\leq (k-1)\tilde{T}}\big\|u^n(t)\big\|_{W^{1,p}(\mt^N)}^q\Big)\\
&\quad\leq \sum_{i=1}^K(\tilde{C}_{\tilde{T}})^i+(\tilde{C}_{\tilde{T}})^K\stred\|u_0\|_{W^{1,p}(\mt^N)}^q\leq \bar{C}\big(1+\stred\|u_0\|_{W^{1,p}(\mt^N)}^q\big),
\end{split}
\end{equation*}
where the constant $\bar{C}$ is independent of $k$  and $n$. Finally, the estimate \eqref{kaka} follows:
\begin{equation*}
\begin{split}
\stred\sup_{0\leq t\leq T}&\|u^n(t)\|_{W^{1,p}(\mt^N)}^q=\stred\max_{k=1,\dots,K}\sup_{(k-1)\tilde{T}\leq t\leq k\tilde{T}}\|u^n(t)\|_{W^{1,p}(\mt^N)}^q\\
&\leq\sum_{k=1}^K\stred\sup_{(k-1)\tilde{T}\leq t\leq k\tilde{T}}\|u^n(t)\|_{W^{1,p}(\mt^N)}^q\leq K\bar{C}\big(1+\stred\|u_0\|_{W^{1,p}(\mt^N)}^q\big).
\end{split}
\end{equation*}
%Moreover, since the identity mapping from $W^{1,p}(\mt^N)$ to $L^p(\mt^N)$ is continuous and injective, any Borel set in $W^{1,p}(\mt^N)$ is also Borel in $L^p(\mt^N)$, so its preimage under $u^n,\,n\in\mn,$ is predictable.

We have now all in hand to deduce that the sequence $\{u^n;\,n\in\mn\}$ is bounded in
$$L^q(\Omega;L^\infty(0,T;W^{1,p}(\mt^N)))$$
and therefore has a weak-star convergent subsequence.
Any norm is weakly lower semicontinuous so we get the estimate \eqref{sobolev1} for the limit process $u$. Moreover, since the stochastic convolution has a continuous modification according to \cite[Corollary 3.5]{b1}, the proof is complete.
\end{proof}
\end{prop}

Proof of regularity in higher order Sobolev spaces (order greater than 1) is more complicated as the norm of a superposition does not, in general, grow linearly with the norm of the inner function (cf. Proposition \ref{admissible}, Corollary \ref{admissible1}, Remark \ref{first}).

\begin{prop}[Estimate in $W^{m,p}(\mt^N)$]\label{prop3}
Let $p\in[2,\infty),\,q\in(2,\infty),\,m\in\mn,$ $m\geq 2$. Assume that $u_0\in L^q(\Omega;W^{m,p}(\mt^N))\cap L^{mq}(\Omega;W^{1,mp}(\mt^N))$ and
$$f_\alpha\in C^{m}(\mr)\cap C^{2l-1}(\mr),\;\; |\alpha|\leq 2l-1;\;\quad\sigma_i\in C^m(\mt^N\times\mr),\;\;i=1,\dots,d,$$
have bounded derivatives up to order $m$.
Then the mild solution of \eqref{semilin} belongs to
$$L^q(\Omega;C([0,T];W^{m,p}(\mt^N)))$$
and the following estimate holds true
\begin{equation}\label{sobolev2}
\stred\sup_{0\leq t\leq T}\|u(t)\|_{W^{m,p}(\mt^N)}^q\leq C\big(1+\stred\|u_0\|_{W^{m,p}(\mt^N)}^q+\stred\|u_0\|_{W^{1,mp}(\mt^N)}^{mq}\big).
\end{equation}

\begin{proof}
First, we intend to prove the following estimate for the Picard iterations
\begin{equation}\label{dada}
\stred\sup_{0\leq t\leq T}\|u^n(t)\|_{W^{m,p}(\mt^N)}^q\leq C\big(1+\stred\|u_0\|_{W^{m,p}(\mt^N)}^q+\stred\|u_0\|_{W^{1,mp}(\mt^N)}^{mq}\big),
\end{equation}
with a constant independent of $n$. By induction on $n$, assume that the hypothesis is satisfied for $u^{n-1}$ and compute the estimate for $u^n$. The following arguments and calculations are mostly similar to those in Proposition \ref{prop2}. Recall that according to \eqref{kaka}, we have
\begin{equation}\label{koko}
\stred\sup_{0\leq t\leq T}\|u^n(t)\|_{W^{1,mp}(\mt^N)}^{mq}\leq C\big(1+\stred\|u_0\|_{W^{1,mp}(\mt^N)}^{mq}\big),\qquad\forall n\in\mn.
\end{equation}
Let us consider the restrictions of the operators $\mathcal{S}_p(t),\,t\geq 0,$ to the Sobolev space $W^{m,p}(\mt^N)$ and denote them by $\mathcal{S}_{m,p}(t),\,t\geq0$. By \cite[Theorem V.2.1.3]{ama}, we obtain a strongly continuous semigroup of on $W^{m,p}(\mt^N)$ generated by part of $\mathcal{A}_p$ in $W^{m,p}(\mt^N)$. We denote the generator by $\mathcal{A}_{m,p}$. It follows
\begin{equation*}
\begin{split}
\stred\sup_{0\leq t\leq T}\|\mathcal{S}_{p}(t) u_0\|_{W^{m,p}(\mt^N)}^q&=\stred\sup_{0\leq t\leq T}\|\mathcal{S}_{m,p}(t) u_0\|_{W^{m,p}(\mt^N)}^q\\
&\leq C\,\stred\|u_0\|_{W^{m,p}(\mt^N)}^q.
\end{split}
\end{equation*}
As above, we employ the following bounded operator: let $\delta=\frac{2l-1}{2l}$
\begin{equation*}
\begin{split}
\mathcal{B}_{m,p}\;:\;W^{m,p}(\mt^N;\mr^\gamma)&\longrightarrow W^{m,p}(\mt^N)\\
\{z_\alpha\}_{\alpha=1}^\gamma&\longmapsto (-\mathcal{A}_{p})^{-\delta}\sum_{\substack{|\alpha|\leq 2l-1\\a_\alpha\neq0}}a_\alpha\totdif^\alpha z_\alpha,
\end{split}
\end{equation*}
so
\begin{equation*}
\begin{split}
\stred&\sup_{0\leq t\leq T}\bigg\|\int_0^t\mathcal{S}_{p}(t-s)F\big(u^{n-1}(s)\big)\,\dif s\bigg\|_{W^{m,p}(\mt^N)}^q\\
&\quad\leq \stred\sup_{0\leq t\leq T}\bigg(\int_0^t\Big\|(-\mathcal{A}_{p})^\delta\mathcal{S}_{p}(t-s)\mathcal{B}_{m,p}f\big(u^{n-1}(s)\big)\Big\|_{W^{m,p}(\mt^N)}\dif s\bigg)^q\\
&\quad\leq \stred\sup_{0\leq t\leq T}\bigg(\int_0^t\Big\|(-\mathcal{A}_{m,p})^\delta\mathcal{S}_{m,p}(t-s)\mathcal{B}_{m,p}f\big(u^{n-1}(s)\big)\Big\|_{W^{m,p}(\mt^N)}\dif s\bigg)^q\\
&\quad\leq C\,\stred\sup_{0\leq t\leq T}\bigg(\int_0^t\frac{1}{(t-s)^\delta}\big\|f\big(u^{n-1}(s)\big)\big\|_{W^{m,p}(\mt^N)}\dif s\bigg)^q\\
&\quad\leq CT^{q(1-\delta)}\stred\sup_{0\leq t\leq T}\big\|f\big(u^{n-1}(t)\big)\big\|_{W^{m,p}(\mt^N)}^q.
\end{split}
\end{equation*}
And for the stochastic term
\begin{equation*}
\begin{split}
\big\|\sigma\big(u^{n-1}(s)\big)&\big\|_{\gamma(\mathfrak{U};W^{m,p}(\mt^N))}^q=\bigg(\stred\Big\|\sum_{i=1}^d\xi_i\,\sigma_i\big(\cdot,u^{n-1}(s,\cdot)\big)\Big\|_{W^{m,p}(\mt^N)}^2\bigg)^\frac{q}{2}\\
&\leq \bigg(\stred\Big\|\sum_{i=1}^d\xi_i\,(-\mathcal{A}_{p})^\frac{m}{2l}\sigma_i\big(\cdot,u^{n-1}(s,\cdot)\big)\Big\|_{L^p(\mt^N)}^p\bigg)^\frac{q}{p}\\
&= C\bigg(\int_{\mt^N}\Big(\sum_{i=1}^d\big|(-\mathcal{A}_{p})^\frac{m}{2l}\sigma_i\big(y,u^{n-1}(s,y)\big)\big|^2\Big)^\frac{p}{2}\dif y\bigg)^\frac{q}{p}\\
&\leq C\sum_{i=1}^{d}\big\|\sigma_i\big(\cdot,u^{n-1}(s,\cdot)\big)\big\|^q_{W^{m,p}(\mt^N)}
\end{split}
\end{equation*}
hence
\begin{equation*}
\begin{split}
\stred\sup_{0\leq t\leq T}\bigg\|\int_0^t\mathcal{S}_{p}(t-s)&\sigma\big(u^{n-1}(s)\big)\dif W(s)\bigg\|_{W^{m,p}(\mt^N)}^q\\
&\leq CT^{\frac{q}{2}-1}\,\stred\int_0^T\big\|\sigma\big(u^{n-1}(s)\big)\big\|_{\gamma(\mathfrak{U};W^{m,p}(\mt^N))}^q\dif s\\
&\leq CT^{\frac{q}{2}-1}\sum_{i=1}^d\stred\int_0^T\big\|\sigma_i\big(\cdot,u^{n-1}(s,\cdot)\big)\big\|^q_{W^{m,p}(\mt^N)}\dif s\\
&\leq CT^\frac{q}{2}\sum_{i=1}^d\stred\sup_{0\leq t\leq T}\big\|\sigma_i\big(\cdot,u^{n-1}(t,\cdot)\big)\big\|^q_{W^{m,p}(\mt^N)}.
\end{split}
\end{equation*}
We conclude
\begin{equation*}
\begin{split}
\stred\sup_{0\leq t\leq T}\|u^n(t)&\|_{W^{m,p}(\mt^N)}^q\leq C\,\stred\|u_0\|_{W^{m,p}(\mt^N)}^q\\
&+CT^{q(1-\delta)}\stred\sup_{0\leq t \leq T}\big\|f\big(u^{n-1}(t)\big)\big\|_{W^{m,p}(\mt^N)}^q\\
&+CT^\frac{q}{2}\sum_{i=1}^d\stred\sup_{0\leq t \leq T}\big\|\sigma_i\big(\cdot,u^{n-1}(t,\cdot)\big)\big\|_{W^{m,p}(\mt^N)}^q.
\end{split}
\end{equation*}
Applying Proposition \ref{admissible}, Corollary \ref{admissible1} and \eqref{koko} we obtain
\begin{equation*}
\begin{split}
&\stred\sup_{0\leq t\leq T}\|u^n(t)\|_{W^{m,p}(\mt^N)}^q\leq C\,\stred\|u_0\|_{W^{m,p}(\mt^N)}^q+C\big(T^{q(1-\delta)}+T^\frac{q}{2}\big)\\
&\qquad\times\bigg(1+\stred\sup_{0\leq t\leq T}\!\|u^{n-1}(t)\|_{W^{1,mp}(\mt^N)}^{mq}+\stred\sup_{0\leq t\leq T}\!\|u^{n-1}(t)\|_{W^{m,p}(\mt^N)}^q\bigg)\\
&\!\leq C\,\stred\|u_0\|_{W^{m,p}(\mt^N)}^q+C\big(T^{q(1-\delta)}+T^\frac{q}{2}\big)\\
&\qquad\times\bigg(1+\stred\|u_0\|_{W^{1,mp}(\mt^N)}^{mq}+\stred\sup_{0\leq t\leq T}\|u^{n-1}(t)\|_{W^{m,p}(\mt^N)}^q\bigg).
\end{split}
\end{equation*}
Let $T$ satisfy the following condition
$$C_T=C\big(T^{q(1-\delta)}+T^\frac{q}{2}\big)<1$$
and define $K_n=\stred\sup_{0\leq t\leq T}\|u^n(t)\|_{W^{m,p}(\mt^N)}^q,\,n\in\mn_0,$ $L_0=\stred\|u_0\|_{W^{1,mp}(\mt^N)}^{mq}.$
Then we have
$$K_n\leq C\,\stred\|u_0\|_{W^{m,p}(\mt^N)}^q+C_T\big(1+L_0+K_{n-1}\big)$$
hence inductively in $n$
\begin{equation*}
\stred\sup_{0\leq t\leq T}\|u^n(t)\|_{W^{m,p}(\mt^N)}^q\leq \tilde{C}_T\big(1+\stred\|u_0\|_{W^{m,p}(\mt^N)}^q+\stred\|u_0\|_{W^{1,mp}(\mt^N)}^{mq}\big),
\end{equation*}
where the constant does not depend on $n$. Therefore  \eqref{dada} follows under the additional hypothesis upon $T$. However, this condition can be removed by the same approach as in Proposition \ref{prop2}.
%Moreover, since the identity mapping from $W^{m,p}(\mt^N)$ to $L^p(\mt^N)$ is continuous and injective, any Borel set in $W^{m,p}(\mt^N)$ is also Borel in $L^p(\mt^N)$, so its preimage under $u^n,\,n\in\mn,$ is predictable. 

Similarly to Proposition \ref{prop2} we deduce that the sequence $\{u^n;\,n\in\mn\}$ is bounded in
$$L^q(\Omega;L^\infty(0,T;W^{m,p}(\mt^N)))$$
and therefore \eqref{sobolev2} holds true. Existence of a continuous modification follows again from \cite[Corollary 3.5]{b1}.
\end{proof}
\end{prop}

\begin{proof}[Proof of Theorem \ref{smooth1}]
If $m=1$ the proof is an immediate consequence of Propositions \ref{prop1} and \ref{prop2}. The case $m\geq2$ follows from Propositions \ref{prop1}, \ref{prop2} and \ref{prop3}.
\end{proof}

%In order to weaken the smoothness assumptions on coefficients we make use
%
%As a consequence of our main result we obtain a smooth $C^{k,\lambda}(\mt^N)$-valued solution to \eqref{semilin}. This result does not depend on the space dimension. However, since the idea is to work in the context of power scale generated by the strongly elliptic operator $-\mathcal{A}$ and apply the Sobolev embedding theorem, where dimension plays an important role, we make use of stochastic integration in Banach spaces (see \cite{b2}, \cite{ondrejat3}), i.e. $W^{m,p}(\mt^N)$, so that integrability indexes help to compensate the lack of differentiability.

\begin{proof}[Proof of Corollary \ref{smooth2}]
Let $m=k+1$. According to Theorem \ref{smooth1} there exists a solution of \eqref{semilin} which belongs to 
$$L^q(\Omega;C([0,T];W^{m,p}(\mt^N))),\qquad\forall p\in[2,\infty).$$
If $p>N$, then according to the Sobolev embedding theorem, the space $W^{m,p}(\mt^N)$ is continuously embedded in $C^{k,\lambda}(\mt^N)$ for $\lambda\in(0,1-N/p)$. %which implies also predictability of $u$ as a $C^k(\mt^N)$-valued process. 
Hence the assertion follows.
% for $p>\max\{N,2l\}$. Nevertheless, the case of $p\in[2,\max\{N,2l\}]$ is an immediate consequence so the proof is complete.
\end{proof}

\subsection*{Acknowledgment}
The author is indebted to Arnaud Debussche and Jan Seidler for valuable consultations and comments.

\end{document}